\documentclass{article}
\usepackage[journal=JST, lang=british]{ems-journal}

\numberwithin{equation}{section}

\newtheorem{theorem}{Theorem}[section]

\newtheorem{lemma}[theorem]{Lemma}
\newtheorem{proposition}[theorem]{Proposition}

\theoremstyle{definition}

\newtheorem{example}[theorem]{Example}
\newtheorem{remark}[theorem]{Remark}

\newcommand{\D}{\mathrm{d}}
\newcommand{\e}{\mathrm{e}}
\newcommand{\N}{\mathbb{N}}
\newcommand{\R}{\mathbb{R}}
\newcommand{\Z}{\mathbb{Z}}
\newcommand{\BB}{\mathcal{B}}
\newcommand{\KK}{\mathcal{K}}
\newcommand{\OO}{\mathcal{O}}
\newcommand{\sumoplus}{\!\raisebox{1.5ex}{{\scriptsize $\oplus$}}}

\begin{document}

\title{Local perturbations of potential well arrays}

\emsauthor{1}{
    \givenname{Pavel}
    \surname{Exner}
    \mrid{} % Replace with your actual MR Author ID if available
    \orcid{} % Replace with your actual ORCID if available
}{P.~Exner}

\emsauthor{2}{
    \givenname{David}
    \surname{Spitzkopf}
    \mrid{} % Leave empty if not available
    \orcid{} % Leave empty if not available
}{D.~Spitzkopf}
\Emsaffil{1}{
    \department{1}{Doppler Institute for Mathematical Physics and Applied Mathematics}
	\organisation{1}{Czech Technical University}
	\rorid{1}{}
	\address{1}{B\v rehov\'a 7, 11519 Prague, Czechia}
	\department{2}{Department of Theoretical Physics}
	\organisation{2}{NPI, Academy of Sciences}%
	\address{2}{25068 \v{R}e\v{z} near Prague}%
	\country{2}{Czechia}
	\affemail{2}{exner@ujf.cas.cz}}

\Emsaffil{2}{
	\department{1}{Department of Theoretical Physics}
	\organisation{1}{NPI, Academy of Sciences}%
	\address{1}{25068 \v{R}e\v{z} near Prague, Czechia}%
	\department{2}{Faculty of Mathematics and Physics}
	\organisation{2}{Charles University}%
	\address{2}{V Hole\v{s}ovi\v{c}k\'ach 2}%
	\zip{2}{18000}
	\city{2}{Prague}
	\country{2}{Czechia}
	\affemail{2}{spitzkopf@ujf.cas.cz}
}

\classification[2010]{81Q37, 35J10, 35P15}

\keywords{Schr\"odinger operators, periodic potential array, geometrically induced discrete spectrum}

\begin{abstract}
We consider an equidistant array of disjoint potential wells in $\R^\nu,\: \nu\ge 2$, built over a straight line, and show that, under a restriction on the potential support aspect ratio, a perturbation consisting of longitudinal shifts of a finite number of them preserving the disjointness gives rise to a nonempty discrete spectrum below the threshold of the lowest spectral band.
\end{abstract}

\maketitle

%%%%%%%%%%%%%%%%%%%%%%
\section{Introduction}
\label{s:intr}

At a time, the existence and properties of the discrete spectrum coming from local perturbations of a periodic Schr\"odinger operator was a topic of intense interest. The one-dimensional situation was discussed, for instance, in \cite{Zh71, Fa89}, or \cite{GS93} where one can find references to an earlier work by Firsova and Rofe-Beketov. Its higher-dimensional counterparts with unperturbed potential being `fully periodic', i.e. having a bounded elementary cell, were analyzed in \cite{DH86, FK98}. A more general setting appeared in \cite{ADH89} where the unperturbed potential even need not be periodic and the perturbation was supposed to be sign-definite.

Recently a similar problem appeared again in connection with the investigation of geometric perturbations of \emph{soft waveguides}, cf.~\cite{EV24} and reference therein. In \cite{Ex24} an infinite array of disjoint and rotationally symmetric potential wells in dimensions $\nu=2,3$ was considered, the centers of which lay on a curve at equal arcwise distances; it was shown that if the curve is not straight but it is straight outside a compact, the corresponding Schr\"odinger operator has a nonempty discrete spectrum. Here we consider such an array, in any dimension $\nu\ge 2$ and without any symmetry, which remains straight and is subject to a local perturbation consisting of shifting a finite number of the wells preserving their disjoint character. This perturbation is much weaker than the one considered \cite{Ex24}; we are going to show that it can nevertheless produce eigenvalues at the bottom of the spectrum. Rather than an additive perturbation in the spirit of the older papers mentioned above, the effect may be regarded as a sort of substantially refined version of the potential well conspiracy \cite{KS79}.

Last not least, let us express our gratitude to the anonymous referee for careful reading of the manuscript, pointing out various minor inconsistencies, and suggesting an alternative way to prove Proposition~3.1.

%%%%%%%%%%%%%%%%%%%%%%%%%
\section{Problem setting}
\label{s:setting}

The potential wells in $\R^\nu$, $\nu\ge 2$, we are going to consider refer to a real-valued and \emph{nonnegative} function $V\in L^p(\Sigma_{\rho,R}(0))$, where $p=2$ for $\nu=2,3$ and $p>\frac12\nu$ for $\nu\ge 4$, and $\Sigma_{\rho,R}(0)=(-\rho,\rho)\times B_R(0)$ with $\rho,R>0$, $B_R(0)\subset\R^{\nu-1}$ being a ball centered at the origin. They will be obtained by shifts: given an infinite discrete set $Y=\{y_i\}\subset\R\times\R^{\nu-1}$, we denote by $V_j:\, x\mapsto V(x-y_j)$ the potential supported in the set $\overline{\Sigma_{\rho,R}(y_i)}$, where $\Sigma_{\rho,R}(y_i):= \Sigma_{\rho,R}(0)+y_i$. We do not require any particular symmetry of the `individual' potential, however, we always assume that the supports do not overlap; in the case we are most interested in, when the points of $Y$ are on a line identified with the first axis, it means that $\mathrm{dist}(y_i,y_j)> 2\rho$ if $i\ne j$. Excluding, of course, the trivial case $V=0$, our goal in this paper is to discuss Schr\"odinger operators
 % ------------- %
 \begin{equation} \label{hamilt}
H_{V,Y} = -\Delta - \sum_j V_j(x), \quad D(H_{V,Y})=H^2(\R^\nu),
 \end{equation}
 % ------------- %
for a particular family of the sets $Y$; we will use the shorthand $-V_Y$ for the potential term on the right-hand side of \eqref{hamilt}, cf.~Figure~\ref{fig:array}.
 % ------------- %
 \begin{figure}[t]
 \centering
 \includegraphics[scale=0.2]{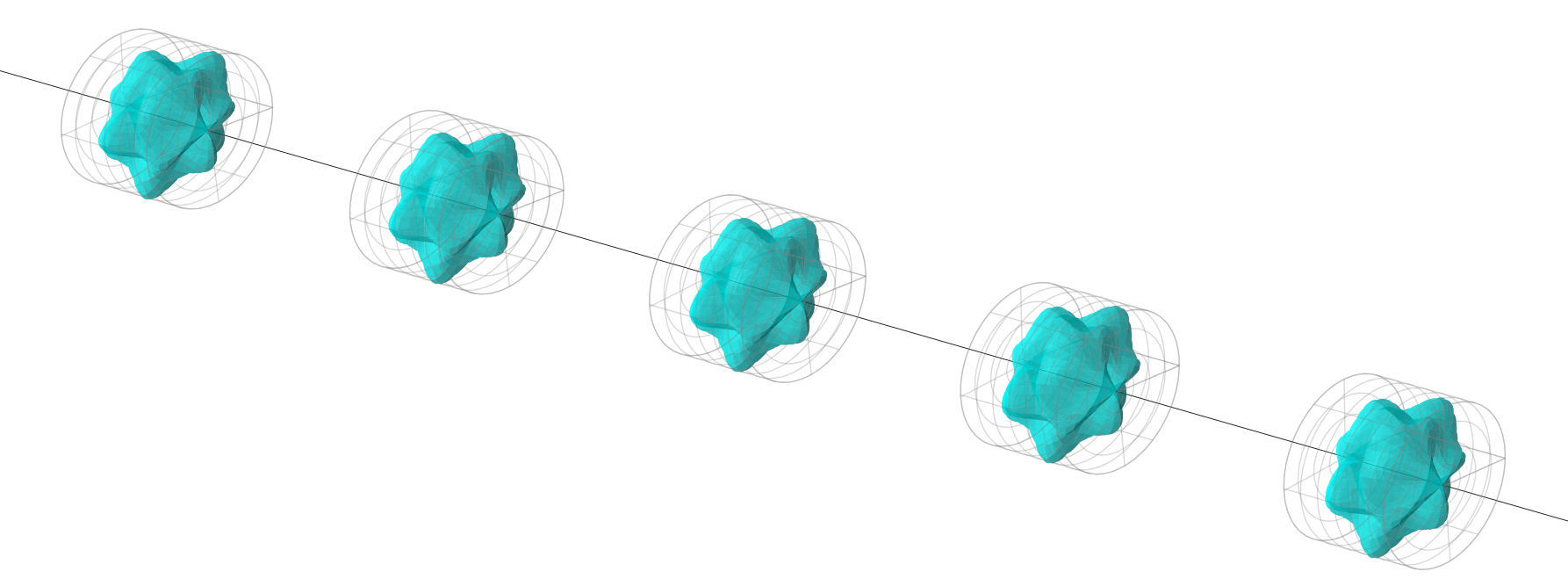}
 \caption{The potential array support with the indicated size restrictions of individual wells}
 \label{fig:array}
 \end{figure}
 % ------------- %
It is not difficult to check that $H_{V,Y}$ is self-adjoint and bounded from below, cf.~Proposition~\ref{prop:spess}.

As in \cite{Ex24}, the `unperturbed' system refers to the situation when the well centers form a straight equidistant array, $Y_0=\{(ja,0)\in\R\times\R^{\nu-1}:\: j\in\Z\}$ of a spacing $a>2\rho$. In the mentioned paper, however, we considered $\nu=2,3$, and what is more important, the potentials were supposed to be spherically symmetric and the points of $Y$ to lie on a non-straight but asymptotically straight curve with a fixed arcwise distance; here our primary focus is on arrays of points at the same line as $Y_0$ differing from the latter by changing positions of a finite number of them.

Recall that the spectrum of a single-well operator, which we denote as $H_V$, depends on the dimension. We have always $\sigma_\mathrm{ess}(H_V)=[0,\infty)$, however, while the discrete spectrum is always nonempty if $\nu=2$ \cite{Si76}, it is well known that in the higher-dimensional situation the existence of eigenvalues requires a critical strength, e.g., $\int_0^\rho V(r)\,r\D r > 1$ for $\nu=3$ \cite{JP51, Ba52}.

As we have said, the departing point of our discussion is the straight and equidistant array, $Y=Y_0$. In this case the spectrum of operator is purely essential; its properties are collected in the following proposition:
 % ------------- %
 \begin{proposition} \label{prop:straightline}
The operator $H_{V,Y_0}$ is self-adjoint and bounded from below. We have $\sigma(H_{V,Y_0})= \sigma_\mathrm{ess}(H_{V,Y_0})$, where the spectrum contains the interval $[0,\infty)$ and $\inf\sigma(H_{V,Y_0})<0$. In the negative part, the spectrum may or may not have gaps; their number is finite and does not exceed $\#\sigma_\mathrm{disc}(H_{V})$. This bound is saturated for the spacing $a$ large enough if $\nu=2$, in the case $\nu\ge 3$ there may be one gap less which happens if the potential is weak, i.e. for $H_{\lambda V,Y_0}$ with $\lambda>0$ sufficiently small.
 \end{proposition}
 % ------------- %
\begin{proof}
By assumption the function $V$ belongs to $L^p(\Sigma_{\rho,R}(0))$ for any $p>\frac12\nu$ and $\nu\ge 2$, which means that the potential $V_{Y_0}$ satisfies
 % ------------- %
\begin{equation*}
\sup_{x\in\R^\nu} \int_{B_1(x)} V_{Y_0}(x')^p\,\D x' < \infty
\end{equation*}
 % ------------- %
and by \cite[Prop.~4.3]{AS82} it belongs to Kato class $\KK_\nu$, then the operator \eqref{hamilt} is self-adjoint and bounded from below by \cite[Thm.~A.2.7]{Si82}. An alternative, more explicit way to check the self-adjointness is to employ Theorem~X.63 of \cite{RS} approximating $H_{V,Y_0}$ by the sequence of cut-off operators
 % ------------- %
\begin{equation*}
H_{V,Y_0}^{(n)} = -\Delta - \sum_{|j|\le n} V_j(x),
\end{equation*}
 % ------------- %
each of which is self-adjoint by virtue of a simple Kato-Rellich-type argument. Instead of points $\pm i$ of the resolvent set in the said theorem we can use $-\kappa^2$ with $\kappa$ large enough because the sequence is non-increasing and $H_{V,Y_0}$ is semibounded as we shall see a little below. By the resolvent identity we get for such a $\kappa$
 % ------------- %
\begin{equation*}
(H_{V,Y_0}+\kappa^2)^{-1} - (H_{V,Y_0}^{(n)}+\kappa^2)^{-1} = (H_{V,Y_0}^{(n)}+\kappa^2)^{-1}\sum_{|j|>n} V_j\: (H_{V,Y_0}+\kappa^2)^{-1}
\end{equation*}
 % ------------- %
which is an operator with the norm not exceeding $2\big(\kappa^2+\inf\sigma(H_{V,Y_0})\big)^{-1}$; we denote it $D_n$. The resolvent kernels involved are positive, so that applying it to an arbitrary $\phi\in L^2(\R^\nu)$, we get the sequence of functions $D_n\phi$, non-increasing and converging pointwise to zero, hence $H_{V,Y_0}^{(n)} \to H_{V,Y_0}$ in the strong resolvent sense as $n\to\infty$. It remains to check that the sequence of
 % ------------- %
\begin{equation*}
A_n:= H_{V,Y_0} - H_{V,Y_0} = \sum_{|j|>n} V_j
\end{equation*}
 % ------------- %
has a densely defined weak graph limit, which is the case because for any $\psi\in C_0^\infty(\R^\nu)$ the sequence $\{A_n\psi\}$ converges weakly to zero.

A detailed proof of the essential spectrum claim was given in \cite{Ex24} for $\nu=2,3$; it carries over to higher $\nu$; here we sketch it only briefly to recall some notions we will need in the following. The inclusion $\sigma(H_{V,Y_0})\supset[0,\infty)$ is easy to check using a Weyl sequence of functions the supports of which are compact and disjoint moving away from the array axis; the negative spectrum is dealt with using Floquet decomposition, $H_{V,Y_0} = \int_\BB^\oplus H_V(\theta)\,\D\theta$, with $\BB=\big[\!-\!\frac{\pi}{a}, \frac{\pi}{a}\big)$ and the fiber $H_V(\theta)$ in $L^2(S_a)$, where $S_a:= J_a\times\R^{\nu-1}$ with $J_a:=\big(-\frac{a}{2}, \frac{a}{2}\big)$ is the periodicity cell, acting as $H_V = -\Delta - V$ on the domain
 % ------------- %
 \begin{align} \label{fiberdom}
D(H_V(\theta)) & = \Big\{ \psi\in H^2(S_a):\: \psi\big(\textstyle{\frac{a}{2}},x_\perp\big) = \e^{i\theta} \psi\big(-\textstyle{\frac{a}{2}},x_\perp\big) \;\; \text{and} \\ & \partial_{x_1}\psi\big(\textstyle{\frac{a}{2}},x_\perp\big) = \e^{i\theta} \partial_{x_1}\psi\big(-\textstyle{\frac{a}{2}},x_\perp\big) \;\;\text{for all}\;\; x_\perp\in\R^{\nu-1}\Big\}\,; \nonumber
 \end{align}
 % ------------- %
it is associated with the quadratic form
 % ------------- %
 \begin{equation} \label{fiberform}
Q_{V,\theta}[\psi] := \int_{S_a} \Big(\big|\big(-i\partial_{x_1} -\textstyle{\frac{\theta}{a}}\big)\psi(x)\big|^2 + |\nabla_{x_\perp}\psi(x)|^2 - V(x)|\psi(x)|^2\Big)\, \D x
 \end{equation}
 % ------------- %
with the domain consisting of $H^1(S_a)$ functions satisfying the periodic boundary conditions, $\psi\big(\frac{a}{2},x_\perp\big) = \psi\big(\!-\!\frac{a}{2},x_\perp\big)$. In view of the compact support of $V$, the negative spectrum of $H_V(\theta)$ consists of at most finite number of eigenvalues which are continuous functions of $\theta$ and their ranges constitute the spectral bands; this implies, in particular, that $H_{V,Y_0}$ is bounded from below. The spectral threshold of $H_{V,Y_0}$ corresponds to $\min_{\theta\in\BB} \inf\sigma(H_V(\theta))$; the corresponding eigenvalue is associated with the eigenfunction $\psi_0$ chosen to be real-valued and positive; its periodic extension denoted by the same symbol is the corresponding generalized eigenfunction of $H_{V,Y_0}$ referring to the threshold. Finally, to show that $\inf\sigma(H_{V,Y_0})<0$ holds always, one can use $\psi_0$ with suitable mollifiers to construct explicit trial functions making the value of the quadratic form \eqref{fiberform} negative.

In the particular case of the mirror-symmetric potential, $V(x_1,x_\perp)=V(-x_1,x_\perp)$, the lower and upper band edges correspond respectively to the center and the endpoint of $\BB$, or equivalently, to operators $H_V^\#$, $\#=\mathrm{N}, \mathrm{D}$, with Neumann and Dirichlet conditions; the eigenfunction $\psi_0$ then corresponds to the lowest eigenvalue of $H_V^\mathrm{N}$.
\end{proof}
 % ------------- %
 \begin{remark} \label{rem:critical}
If the potential is not sign-definite, the claim remains true except that $\inf\sigma(H_{V,Y_0})<0$ requires additionally $\int_{\Sigma_{\rho,R}(0)} V(x)\,\D x>0$. The existence of the negative spectrum for weak potentials in the dimensions \mbox{$\nu\ge 3$} does contradict the need of a critical strength to achieve \mbox{$\inf\sigma(H_V)<0$;} note that the spectral threshold converges in its absence to zero as $a\to\infty$. One naturally expects the spectrum of $H_{V,Y_0}$ to be absolutely continuous. The question about conditions under which this happens in the present case where the elementary cell, the slab $S_a$, is unbounded, is of independent interest, however, we will not pursue it here because it has no importance for the effect we are investigating.
 \end{remark}
 % ------------- %

%%%%%%%%%%%%%%%%%%%%%%%
\section{Preliminaries}
\label{s:prelim}

As indicated, the subject of our interests are perturbations of $H_{V,Y_0}$ which are \emph{local} in the sense that there is bounded set $P$ such that the symmetric difference $\Sigma(V_Y,V_{Y_0}):=\mathrm{supp}\,V_Y\,\Delta\, \mathrm{supp}\,V_{Y_0}\subset P$. To begin with, we have to check that the self-adjointness is preserved and so is the essential spectrum of the operator; note that for the purpose of this paper we just need that the essential spectrum \emph{threshold} remains the same.
 % ------------- %
 \begin{proposition} \label{prop:spess}
The operator $H_{V,Y}$ is self-adjoint on $H^2(\R^\nu)$ for any local perturbation and $\sigma_\mathrm{ess}(H_{V,Y})=\sigma_\mathrm{ess}(H_{V,Y_0})$.
 \end{proposition}
 % ------------- %
\begin{proof}
Concerning the self-adjointness and semiboundedness, the Kato class argument from the proof of Proposition~\ref{prop:straightline} applies to local perturbations $V_Y$ of $V_{Y_0}$ as well. To prove the spectral stability, we employ Weyl's criterion. To prove the inclusion $\sigma_\mathrm{ess}(H_{V,Y}) \supset \sigma_\mathrm{ess}(H_{V,Y_0})$ one can proceed as in \cite{Ex24} and construct to any $\lambda\in\sigma_\mathrm{ess}(H_{V,Y_0})$ a sequence $\{\psi_n\}\subset H^2(\R^\nu)$ which converges weakly to zero and $\|(H_{V,Y}-\lambda)\psi_n\|\to 0$ as $n\to\infty$. By assumption such sequences exist for $H_{V,Y_0}$ and one can construct their elements explicitly as products of the generalized eigenfunction of $H_{V,Y_0}$ corresponding to~$\lambda$ (given by the direct integral decomposition mentioned in the proof of Proposition~\ref{prop:straightline}) and suitable mollifiers in the $x_1$ variable; since the perturbation is local, in contrast to \cite{Ex24} we do not need the transverse parts of the mollifiers. Without loss of generality we may suppose that the functions $\psi_n$ have compact and pairwise disjoint supports; then we can use them to construct a Weyl sequence supported in a halfspace where $V_Y$ and $V_{Y_0}$ coincide.

The opposite inclusion follows again from Weyl's criterion, by which to any number $\lambda\in\sigma_\mathrm{ess}(H_{V,Y})$ a sequence $\{\psi_n\}\subset H^2(\R^\nu)$, weakly convergent to zero, for which $\|(H_{V,Y}-\lambda)\psi_n\|\to 0$ as $n\to\infty$. Now we do not know its explicit form, but we can check that $\{\psi_n\}$ is a Weyl sequence for $H_{V,Y_0}$ as well. Indeed, we have
 % ------------- %
\begin{equation*}
\|(H_{V,Y_0}-\lambda)\psi_n\| \le \|(H_{V,Y}-\lambda)\psi_n\| + \|(V_{Y_0}-V_Y)\psi_n\|.
\end{equation*}
 % ------------- %
The first term on the right-hand side tend to zero as $n\to\infty$ by assumption, for the second one we use the fact that the functions $\psi_n$ are continuous, then $\psi_n \overset{w}\longrightarrow 0$ implies that on any compact $\psi_n$ converges to zero pointwise being uniformly bounded \cite[Prop.~19.3.1]{Se71}. It follows that
 % ------------- %
\begin{equation*}
\|(V_{Y_0}-V_Y)\psi_n\| = \int_{\Sigma(V_Y,V_{Y_0})} |V_Y(x)-V_{Y_0}(x)|^2\, |\psi_n(x)|^2\, \D x \to 0
\end{equation*}
 % ------------- %
as $n\to\infty$ holds by dominated convergence theorem, taking into account the fact that $V_Y-V_{Y_0} \in L^2$ on $\Sigma(V_Y,V_{Y_0})$; this concludes the proof.
\end{proof}
 % ------------- %
\begin{proof} [Second proof]
There is alternative way to prove the claim\footnote{We thank the referee for suggesting the idea.} based on the following result: let $H_0$ be self-adjoint and $W$ be $H_0$-bounded with the bound less than one so that $H=H_0-W$ is also self-adjoint. Let further $\{P_n\}$ be a family of projections such that $\|WP_n^\perp\|\to 0$ as $n\to\infty$. If $P_n(H_0+i)^{-1}$ is compact for all $n\ge 1$, then $(H+i)^{-1}-(H_0+i)^{-1}$ is compact. Indeed, we have
 % ------------- %
 $$
(H+i)^{-1}-(H_0+i)^{-1} = (H+i)^{-1}WP_n^\perp(H_0+i)^{-1} + (H+i)^{-1}WP_n(H_0+i)^{-1}.
 $$
 % ------------- %
The operator $(H+i)^{-1}W$ is bounded because its adjoint is bounded, and consequently, the second term on the right-hand side is compact for all $n$. The first one converges to zero in the norm, so the norm limit of the whole expression is that of the second term which is naturally compact. In our case $H_0=H_{V,Y_0}$ and $W=V_{Y_0} - V_{Y}$ which is compactly supported. Choosing, for instance, $\mathrm{Ran}\,P_n = L^2(B_n(0))$, we have $WP_n^\perp=0$ for all $n$ large enough. At the same time, we know that $D(H_{V,Y_0})=H^2(\R^\nu)$, hence the compactness of $P_n(H_{V,Y_0}+i)^{-1}$ follows from the Rellich-Kondrashov theorem \cite[Theorem~{6.3}]{AF03}.
\end{proof}
 % ------------- %

The second preliminary item concerns the method to be used, based on the Birman-Schwinger principle; for a rich bibliography concerning this remarkable tool we refer to \cite{BEG22}. The main object is the family of operators
 % ------------- %
 \begin{equation} \label{BSop}
K_{V,Y}(z) := V_Y^{1/2} (-\Delta-z)^{-1} V_Y^{1/2}
 \end{equation}
 % ------------- %
in $L^2(\R^\nu)$ parametrized by $z\in\mathbb{C}\setminus\R_+$; for our purpose $z=-\kappa^2$ with $\kappa>0$ will be important. By assumption about the potential, the operator $K_{V,Y}(-\kappa^2)$ is positive and maps $L^2(\mathrm{supp}\,V_Y) \to L^2(\mathrm{supp}\,V_Y)$. Since $L^2(\mathrm{supp}\,V_Y) = \sum_j{\!\sumoplus}\: L^2(\Sigma_{\rho,R}(y_j))$, one can regard Birman-Schwinger operator \eqref{BSop} as matrix integral operator with the entries
 % ------------- %
 \begin{equation} \label{BSmatrix}
K_{V,Y}^{(i,j)}(-\kappa^2) := V_i^{1/2} (-\Delta+\kappa^2)^{-1} V_j^{1/2}
 \end{equation}
 % ------------- %
mapping $L^2(\Sigma_{\rho,R}(y_j))$ to $L^2(\Sigma_{\rho,R}(y_i))$. The Birman-Schwinger principle allows us to determine eigenvalues of $H_{V,Y}$ by inspection of those of $K_{V,Y}(-\kappa^2)$:
 % ------------- %
\begin{proposition} \label{prop:BS}
$z\in\sigma_\mathrm{disc}(H_{V,Y})$ holds if and only if $\,1\in\sigma_\mathrm{disc}(K_{V,Y}(z))$ and the dimensions of the corresponding eigenspaces coincide. The
operator $K_{V,Y}(-\kappa^2)$ is bounded for any $\kappa>0$ and the function $\kappa \mapsto K_{V,Y}(-\kappa^2)$ is continuously decreasing in $(0,\infty)$ with $\lim_{\kappa\to\infty}\|K_{V,Y}(-\kappa^2)\|=0$.
\end{proposition}
 % ------------- %
\begin{proof}
The first claim is a particular case of a more general and commonly known result, see, e.g., \cite{BGRS97}. Using the explicit form of $(-\Delta-z)^{-1}$ kernel mentioned above, one can check that each of the operators \eqref{BSmatrix} is Hilbert-Schmidt. This is obvious for $i\ne j$ because the kernel is bounded for $|x_1-x'_1|\ge a-2\rho$ and in view of the compact potential supports we have $V_j\in L^2(\Sigma_{\rho,R}(y_j))$ even for $\nu>3$. As for the diagonal entries, in low dimensions one can employ Sobolev inequality \cite[Sec.~IX.4]{RS}; for $\nu=3$ it applies directly, for $\nu=2$ we use the fact that $K_0(\kappa|x|)\le \frac{c}{\kappa|x|}$ holds for some $c>0$ on the potential support. For $\nu>3$ this argument no longer works, but we can get an even stronger result in a different way. It is enough to note that $K_{V,Y}^{(j,j)}(-\kappa^2)$ is nothing but the Birman-Schwinger operator of single-well Schr\"odinger operator $H_{V_j}$, unitarily equivalent to $H_V$. In view of the compact support, the operator $H_{\lambda V_j}$ has for any $\lambda>0$ a finite number of negative eigenvalues depending monotonously on $\lambda$, hence $K_{V,Y}^{(j,j)}(-\kappa^2)$ has a finite rank not exceeding $\#\sigma_\mathrm{disc}(H_V)$.

The full operator $K_{V,Y}(-\kappa^2) = \sum_{i,j\in\Z} K_{V,Y}^{(i,j)}(-\kappa^2)$ is no longer compact, of course, but it remains to be bounded as a consequence of the decay of the resolvent kernel. Since both $V_Y$ and the resolvent kernel are non-negative, we can estimate the corresponding quadratic form value by
 % ------------- %
\begin{equation*}
(\phi, K_{V,Y}(-\kappa^2)\phi) \le \sum_{i,j\in\Z} \big(|\phi_i|, K_{V,Y}^{(i,j)}(-\kappa^2) |\phi_j|\big) \le \sum_{i,j\in\Z} \|\phi_i\| \|K_{V,Y}^{(i,j)}(-\kappa^2)\| \|\phi_j\|
\end{equation*}
 % ------------- %
for any $\phi\in \sum_j{\!\sumoplus}\: L^2(\Sigma_{\rho,R}(y_j))$, where $\phi_j:= \phi|_{\Sigma_{\rho,R}(y_j)}$. To the series on the right-hand side of the inequality one can apply the matrix version of the Schur test \cite[Sec.~II.1]{Co97}. For $Y=Y_0$ are identically spaced which implies
 % ------------- %
\begin{equation} \label{matrnorm}
\sum_{i\in\Z} \|K_{V,Y_0}^{(i,j)}(-\kappa^2)\| = \sum_{j\in\Z} \|K_{V,Y_0}^{(i,j)}(-\kappa^2)\|
\end{equation}
 % ------------- %
independently of the value of the other index. As in the proof of Proposition~\ref{prop:spess} we use the decay of the resolvent kernel: there is a $c>0$ such that $\|K_{V,Y_0}^{(i,j)}(-\kappa^2)\| \le c\,\e^{-\kappa a|i-j|}$ which means that the series \eqref{matrnorm} converge and since $\|\phi\|^2 = \sum_{j\in\Z} \|\phi_j\|^2$, they provide a bound to $\|K_{V,Y_0}(-\kappa^2)\|$.

Let us next look at the difference $K_{V,Y}(-\kappa^2)-K_{V,Y_0}(-\kappa^2)$. Since the perturbation is local by assumption, there is an $N\in\N$ such that the matrix entries of the difference vanish for indices such that $\min\{|i|,|j|\}>N$. The part with $\max\{|i|,|j|\}\le N$ is a finite sum of bounded operators; it remains to check the boundedness of the parts with $|i|\le N$ and $|j|> N$, and vice versa. For brevity, we put $n_{ij}:= \|K_{V,Y}^{(i,j)}(-\kappa^2) -K_{V,Y_0}^{(i,j)}(-\kappa^2)\|$ and we employ again the exponential decay: since the array $Y$ is equidistant for \mbox{$|j|>N$}, a rough estimate gives $n_{ij}\le c\big(1+\e^{2\kappa aN}\big)\,\e^{-\kappa a|i-j|}$. For a fixed $i$ thus the series $\sum_{|j|>N} n_{ij}$ converges, and the maximum of the finite number of $i$ is finite; we denote it $m_1$. Next we take the sums $\sum_{|i|\le N} n_{ij}$ which in view of the  decay satisfy $\lim_{|j|\to\infty}\sum_{|i|\le N} n_{ij}<\infty$, and as a result, $m_2:=\sup_{|j|>N} \sum_{|i|\le N} n_{ij} <\infty$. By Schur test, the norm of the corresponding operator does not exceed $\sqrt{m_1m_2}$, and the same argument applies if we swap the roles of $i$ and $j$.

The continuity in $\kappa$ follows from the functional calculus and  we have
 % ------------- %
\begin{equation*} %\begin{equation} \label{monotonicity}
\frac{\D}{\D\kappa} (\psi,V_Y^{1/2} (-\Delta+\kappa^2)^{-1}\, V_Y^{1/2}\psi) = -2\kappa (\psi,V_Y^{1/2} (-\Delta+\kappa^2)^{-2}\, V_Y^{1/2}\psi) < 0
\end{equation*} %\end{equation}
 % ------------- %
for any nonzero $\psi\in L^2(\mathrm{supp}\,V_Y)$ which implies, in particular, the norm monotonicity with respect to $\kappa$. Finally, recalling that $H_{\lambda V}$ has for any $\lambda>0$ a finite number of negative eigenvalues and $\#\sigma_\mathrm{disc}(H_{\lambda' V}) \ge \#\sigma_\mathrm{disc}(H_{\lambda V})$ holds for $\lambda'\ge\lambda$, we infer that $\lim_{\kappa\to\infty}\|K_{V,Y}^{(i,i)}(-\kappa^2)\|=0$. Indeed, in the opposite case at least the largest eigenvalue of $K_{V,Y}^{(i,i)}(-\kappa^2)$ would have a positive limit as $\kappa\to\infty$ which is by Birman-Schwinger principle in contradiction with the monotonicity of $\lambda\mapsto \#\sigma_\mathrm{disc}(H_{\lambda V})$. Using further the monotonicity of the resolvent kernel w.r.t. $|x-x'|$ we finally get $\|K_{V,Y}^{(i,j)}(-\kappa^2)\| \le \|K_{V,Y}^{(i,i)}(-\kappa^2)\|$ for $i\ne j$ which completes the proof.
\end{proof}
 % ------------- %

%%%%%%%%%%%%%%%%%%%%%%%%%%%%%%%%%%%%%%%%%%%%%%%%%%%
\section{Eigenvalues at the bottom of the spectrum}
\label{s:bottomev}

As indicated, we are interested in eigenvalues at the bottom of the spectrum of $H_{V,Y}$ due to changes of positions of a finite number of potential wells. Let us first note that some perturbations of that type keep the spectral threshold preserved:
 % ------------- %
\begin{example} \label{ex:puretrans}
Consider a potential which is mirror-symmetric, $V(x_1,x_\perp)=V(-x_1,x_\perp)$, and subjected to \emph{purely transversal shifts}, that is, suppose that the first Cartesian components of $y_j=(ja,y_{j,\perp})$ and $y_j^{(0)}=\big(ja,y_{j,\perp}^{(0)}\big)$ are the same for all $j\in\Z$. Let $H_{V,Y}^\mathrm{N}$ be obtained from $H_{V,Y}$ by imposing the additional Neumann conditions at the planes $\big\{(j+\frac12a,x_\perp):\, j\in\Z, x_\perp\in\R^{\nu-1}\big\}$. We know from the proof of Proposition~\ref{prop:straightline} that under the assumed symmetry the generalized eigenfunction $\psi_0$ corresponding to the spectral threshold of $H_{V,Y_0}$ is the periodic extension of the ground-state eigenfunction of $H_V^\mathrm{N}$ on $S_a$ which means that $\inf\sigma(H_{V,Y_0})=\inf\sigma(H_V^\mathrm{N})$. However, both $H_{V,Y}^\mathrm{N}$ and $H_{V,Y_0}^\mathrm{N}$ are unitarily equivalent to $\sum{\sumoplus}_{\!\!\!\!j\in\Z}\, H_V^\mathrm{N}$, or alternatively, they are unitarily equivalent to each other; the respective transformation consists of shifting the transverse coordinate by $y_{j,\perp}-y_{j,\perp}^{(0)}$ in the $j$th copy of $S_a$. Consequently, their spectra are the same, and combining this fact with Neumann bracketing \cite[Sec.~XIII.15]{RS} we get
 % ------------- %
\begin{equation} \label{infspess}
\inf\sigma(H_{V,Y}) \ge \inf\sigma(H_{V,Y}^\mathrm{N}) = \inf\sigma(H_{V,Y_0}^\mathrm{N}) = \inf\sigma(H_{V,Y_0}).
\end{equation}
 % ------------- %
If $Y$ is now a local perturbation of $Y_0$ of the considered type, we have
 % ------------- %
\begin{equation} \label{infspess2}
\inf\sigma(H_{V,Y_0}) = \inf\sigma_\mathrm{ess}(H_{V,Y_0}) = \inf\sigma_\mathrm{ess}(H_{V,Y})
\end{equation}
 % ------------- %
by Propositions~\ref{prop:straightline} and \ref{prop:spess}, which means that there are no eigenvalues at the bottom of the spectrum of $H_{V,Y}$.

Note that inequality \eqref{infspess} does not require locality of the perturbation and one may wonder whether shifting the well centers transversally, say, to the broken lines $\{(x_1,c|x_1|,0):\, x_1\in\R\}$ with some $c>0$, we are not in conflict with the result of \cite{Ex24} establishing the existence of discrete spectrum for such a bent array of rotationally symmetric potental wells. The answer is no, of course, because while \eqref{infspess} remains valid, the last identity in \eqref{infspess2} can no longer be used: in the tilted array the distances between the neighboring wells increase and, as a result, the threshold of $\sigma_\mathrm{ess}(H_{V,Y})$ increases as well.
\end{example}
 % ------------- %

The situation is different is the perturbation consists of \emph{longitudinal} shifts, that is, $y_{j,\perp}=0$ for all $j\in\Z$ and $y_{j,1}\ne ja$ for a finite (and nonempty) subset of the indices~$j$. Without loss of generality we may suppose that the order of the points remains preserved, $y_{j,1}<y_{j+1,1}$ for all $j\in\Z$.

 % ------------- %
\begin{theorem} \label{thm:local_pert_bs}
In the described situation, $\sigma_\mathrm{disc}(H_{V,Y})\ne\emptyset$ holds provided $R\le \rho\sqrt{\nu-1}$.
\end{theorem}
 % ------------- %
\begin{proof}
Let $\kappa_0$ be the spectral parameter value referring the spectral threshold of unperturbed system, $\inf\sigma(H_{V,Y_0})= -\kappa_0^2$; in view of Proposition~\ref{prop:BS} it is enough to check that there is a $\kappa>\kappa_0$ such that $K_{V,Y}(-\kappa^2)$ has eigenvalue one which happens provided
$\sup\sigma(K_{V,Y}(-\kappa_0^2))>1$. Thus we have to find a trial function $\phi\in L^2(\mathrm{supp}\,V_Y)\subset L^2(\R^\nu)$ such that
 % ------------- %
 \begin{equation} \label{BStrial}
(\phi, K_{V,Y}(-\kappa_0^2)\phi) > \|\phi\|^2\,;
 \end{equation}
 % ------------- %
the left-hand side can be written explicitly as
 % ------------- %
\begin{equation} \label{BSform_expl}
\sum_{i,j\in\Z} \int_{\Sigma_{\rho,R}(y_i)\times \Sigma_{\rho,R}(y_j)}  \overline\phi(x)\, V_i^{1/2}(x)\, R_{\kappa_0}(x,x')\, V_j^{1/2}(x')\,\phi(x')\,\D x\,\D x'.
\end{equation}
 % ------------- %
As usual in such situations we employ the generalized eigenfunction $\psi_0$ associated with the bottom of the essential spectrum of $H_{V,Y_0}$ as the starting point; the corresponding generalized eigenfunction of $K_{V,Y_0}$ referring to the upper edge of its spectrum is $\phi_0:=V_{Y_0}^{1/2}\psi_0$. We will use the symbol $\phi_0$ also for the family
 % ------------- %
 \begin{equation} \label{BSunmolif}
\{\phi_{0,j}\} \in \sum_{j\in\Z}\sumoplus\: L^2(\Sigma_{\rho,R}(y_j)), \quad \phi_{0,j}(\xi+y_j):=\phi_0(\xi)\;\;\text{for}\;\, \xi\in \Sigma_{\rho,R}(0)\,;
 \end{equation}
 % ------------- %
the excuse for this abuse of notation is that it is convenient when we write the expression to be estimated in the form \eqref{BSform_expl}. Recall that the `original' function $\psi_0$ could be chosen real-valued and positive; the same applies \emph{mutatis mutandis} to $\phi_0$. By construction, $\phi_0$ is symmetric with respect to the discrete translations; it may have other symmetries inherited from those of the potential $V$.

The function \eqref{BSunmolif} naturally does not belong to $L^2(\mathrm{supp}\,V_Y)$; to make it a viable trial function, we have as usual to multiply it by a suitable family of mollifiers. To this aim, we will employ the functions
 % ------------- %
 \begin{equation} \label{BSmollif}
h_n(x) = \sum_{i\in\Z} h_{n,i} \chi_{\Sigma_{\rho,R}(y_i)}(x)
 \end{equation}
 % ------------- %
with the coefficients chosen, for instance, as
 % ------------- %
 \begin{equation} \label{mollifcoef}
h_{n,i} = \frac{n^2}{n^2+i^2}\,, \quad i\in\Z\,.
 \end{equation}
 % ------------- %
Before proceeding we have to make sure that the effect of the mollifier can be made arbitrarily small.
 % ------------- %
\begin{lemma} \label{l:BSmollif}
$(h_n\phi_0, K_{V,Y_0}(-\kappa_0^2)h_n\phi_0) - \|h_n\phi_0\|^2=\OO(n^{-2})$ as $n\to\infty$.
\end{lemma}
 % ------------- %
\begin{proof}
The first term of the above difference equals
 % ------------- %
\begin{equation*}
\sum_{i\in\Z} \sum_{j\in\Z} \int_{\Sigma_{\rho,R}(y_i)} \! \D x\, \phi_{0,i}(x) \int_{\Sigma_{\rho,R}(y_j)} \! V_Y^{1/2}(x)\, (-\Delta+\kappa_0^2)^{-1}(x,x')\, V_Y^{1/2}(x') \phi_{0,j}(x')\,\D x'
\end{equation*}
 % ------------- %
To simplify the notation for a moment, denote $K_{V,Y_0}(-\kappa_0^2)$ as $K$ and its matrix operator entries as $K_{ij}$. For $\kappa=\kappa_0$ we have $K\phi_0=\phi_0$ so that $\|h_n\phi_0\|^2 = (h_n\phi_0, h_n K\phi_0)$ and the difference in question equals $(\phi_0,h_n[K,h_n]\phi_0)$, or more explicitly
 % ------------- %
\begin{equation*}
D = \sum_{i,j\in\Z} h_{n,i} \big(h_{n,j}-h_{n,i}\big) M_{ij}\,,
\end{equation*}
 % ------------- %
where $M_{ij}= (\phi_{0,i}, K_{ij}\phi_{0,j})$; recall that the $h_{n,i}$'s are just numbers. The $\phi_{0,i}$'s are shifted copies of the same function and in view of the exponential decay of the resolvent kernel we infer that there is a positive number $c$ such that $|M_{ij}|\le c\,\e^{-\kappa a|i-j|}$. Using the explicit form of the coefficients $h_{n,i}$, we can rewrite the expression as
 % ------------- %
\begin{equation*}
D = 4n^2 \sum_{i,j=0}^\infty \frac{i^2-j^2}{(n^2+i^2)^2(n^2+j^2)}\,M_{ij}\,.
\end{equation*}
 % ------------- %
Passing to the summation over $k=i-j$ and $l=i+j$, we get
 % ------------- %
 \begin{equation*}
D = 4n^2 \sum_{k\in\Z} \sum_{l=1}^\infty \frac{kl}{\big(n^2+\frac14(k+l)^2\big)^2 \big(n^2+\frac14(k-l)^2\big)}\,M_{ij}
\end{equation*}
 % ------------- %
or
 % ------------- %
\begin{equation*}
D = -4n^2 \sum_{k=1}^\infty \sum_{l=1}^\infty \frac{k^2l^2}{\big(n^2+\frac14(k+l)^2\big)^2 \big(n^2+\frac14(k-l)^2\big)^2}\,M_{ij}\,.
\end{equation*}
 % ------------- %
Hence the difference is negative and estimating the denominator from below by $\frac{1}{16}\,l^4n^4$, we arrive at the bound
 % ------------- %
\begin{equation*}
|D| \le \frac{64c}{n^2} \: \sum_{l=1}^\infty l^{-2} \: \sum_{k=1}^\infty k^2\,\e^{-\kappa ak},
\end{equation*}
 % ------------- %
which yields the sought result.
\end{proof}

 % ------------- %
\begin{remark}
The proof of Theorem~2.6 in \cite{Ex24} contains a gap: one has to check that the positivity of the curvature-induced perturbation survives in the limit when the mollifier is removed. Instead of checking this directly, one can replace the cut-off in Lemma~3.7 there by \eqref{BSmollif} with the sequence \eqref{mollifcoef} for which the trial function family converges \emph{pointwise} to the generalized eigenfunction at the threshold; recall that the function $\phi_0^Y$ used there can be unitarily mapped to the $\phi_0$ of the above lemma.
\end{remark}
 % ------------- %

\noindent \emph{Proof of Theorem~\ref{thm:local_pert_bs}, continued:} We want thus to prove that
 % ------------- %
\begin{equation*}
(h_n\phi_0, K_{V,Y}(-\kappa_0^2)h_n\phi_0) - \|h_n\phi_0\|^2 > 0
\end{equation*}
 % ------------- %
holds for all $n$ large enough. Adding and subtracting $(h_n\phi_0, K_{V,Y_0}(-\kappa_0^2)h_n\phi_0)$ to the left-hand side and using the above lemma, this will be true if
 % ------------- %
 \begin{equation} \label{trial_ineq}
\lim_{n\to\infty} (h_n\phi_0, K_{V,Y}(-\kappa^2)h_n\phi_0) - (h_n\phi_0, K_{V,Y_0}(-\kappa^2)h_n\phi_0) > 0.
\end{equation}
 % ------------- %
would hold for $\kappa=\kappa_0$, and certainly if we verify \eqref{trial_ineq} for any $\kappa>0$. As in the proof of Proposition~\ref{prop:BS}, due to the local character of the perturbation we can write the sought relation more explicitly,
 % ------------- %
\begin{equation*}
\lim_{n\to\infty} \:2\!\!\! \sum_{(i,j)\in P_N} h_{n,i}h_{n,j}\, \big(\phi_0, \big[K_{V,Y}^{(i,j)}(-\kappa^2) - K_{V,Y_0}^{(i,j)}(-\kappa^2)\big]\phi_0\big)>0
\end{equation*}
 % ------------- %
for some $N\in\N$, the summation running over $P_N=P_N^< \cup P_N^>$ with
 % ------------- %
\begin{align*}
P_N^<:= & \{(i,j): -N\le i<j\le N\}\,, \\[.2em]
P_N^>:= & \{(i,j): i<N\:\&\:|j|\le N \; \text{or}\; |i|\le N\:\&\:j>N\}\,;
\end{align*}
 % ------------- %
recall that $K_{V,Y}^{(i,i)}(-\kappa^2)= K_{V,Y_0}^{(i,i)}(-\kappa^2)$ and $K_{V,Y}^{(i,j)}(-\kappa^2)=K_{V,Y}^{(j,i)}(-\kappa^2)$. Since the perturbation is local by assumption, for a sufficiently large $N$ the contributions to the sum from the index pairs outside $P_N$ are zero.

The set $P_N^<$ is finite and the series $\sum_{(i,j)\in P_N^>}\, (\phi_0,K_{V,Y_0}^{(i,j)}(-\kappa^2)\phi_0)$ converges absolutely in view of the exponential decay of $|M_{ij}|$ observed in the proof of the previous lemma; the same applies to the series in which $K_{V,Y_0}^{(i,j)}(-\kappa^2)$ is replaced by $K_{V,Y}^{(i,j)}(-\kappa^2)$. This allows us to use the dominant convergence theorem and exchange the limit with the double sum; as a result one has to check positivity of the series
 % ------------- %
 \begin{equation} \label{trial_ineq2}
\sum_{i,j\in\Z}\big(\phi_0, \big[K_{V,Y}^{(i,j)}(-\kappa^2) - K_{V,Y_0}^{(i,j)}(-\kappa^2)\big]\phi_0\big).
\end{equation}
 % ------------- %
To estimate the sum we use their explicit expression of its terms, namely
 % ------------- %
 \begin{align*}
& \int_{\Sigma_{\rho,R}(0)} \int_{\Sigma_{\rho,R}(0)} \phi_0(\xi) V^{1/2}(\xi) \big[ R_\kappa(|y_i-y_j +\xi-\xi'|) - R_\kappa(|y_i^{(0)}\!-y_j^{(0)}\! +\xi-\xi'|) \big] \\[.3em] & \qquad \times V^{1/2}(\xi') \phi_0(\xi') \,\D\xi\,\D\xi' \\[.3em]
& = \int_{\Sigma_{\rho,R}(0)} \int_{\Sigma_{\rho,R}(0)} \phi_0(\xi) V^{1/2}(\xi) \big[ R_\kappa(|(i-j)an_1 + \eta_{ij}n_1 +\xi-\xi'|) \\[.3em] & \qquad - R_\kappa(|(i-j)an_1+\xi-\xi'|) \big]\, V^{1/2}(\xi') \phi_0(\xi') \,\D\xi\,\D\xi'
 \end{align*}
 % ------------- %
where $n_1:=(1,0)\in \R\times\R^{\nu-1}$, and furthermore, abusing the notation, we have replaced $R_\kappa(x,x')$ by $R_\kappa(|x-x'|)$\ and introduced the symbol $\eta_{ij}:= (y_i-y_i^{(0)} - y_j+y_j^{(0)})_1$ for the perturbation shift along the first axis. The latter can be expressed through relative shifts of the neighboring points, $\delta_j:= (y_{j+1}-y_j)_1-a$; we obviously have
 % ------------- %
\begin{equation*}
\eta_{ij} = \sum_{r=\min(i,j)}^{\max(i,j)-1} \delta_r
\end{equation*}
 % ------------- %
and we note that with respect to the summation indices appearing in \eqref{trial_ineq2}, this quantity has zero mean. Indeed, by rearrangements we get
 % ------------- %
\begin{align*}
\sum_{i,j\in\Z} \eta_{ij} & = \sum_{l=1}^\infty \Big( \sum_{j\in\Z}\, \sum_{r=j-l}^{j-1} \delta_r + \sum_{j\in\Z} \sum_{r=j}^{j+l-1} \delta_r \Big) = \sum_{l=1}^\infty \sum_{j\in\Z}\, \sum_{r=j-1}^{j+l-1} \delta_r \\[.3em]
& = \sum_{l=1}^\infty \sum_{r\in\Z}\, \sum_{j=r-l+1}^{r+l} \delta_r = \sum_{l=1}^\infty 2l\,\sum_{r\in\Z} \delta_r = 0,
\end{align*}
 % ------------- %
because the last sum vanished due the locality of the perturbation. This suggests to employ Jensen inequality to estimate the expression \eqref{trial_ineq2}. We cannot use the convexity of $R_\kappa$ directly, however, because the perturbation enters its argument through the Euclidean distance between the integration variables in the above expression. Fortunately, there is a region where the convexity is preserved:
 % ------------- %
\begin{lemma} \label{l:comp_convex}
The function $x\mapsto R_\kappa\big(\sqrt{(b+x)^2+c^2}\,\big)$ with $b>0$ and a real $c$ is strictly convex in $(0,\infty)$ for any $\kappa>0$ provided that $c\le b\sqrt{\nu-1}$.
\end{lemma}
 % ------------- %
\begin{proof}
Given that $R_\kappa(\cdot)$ is convex and decreasing, the strict convexity of its composition with $d:\,z(x)=\sqrt{(b+x)^2+c^2}$ is equivalent to the inequality
 % ------------- %
\begin{equation*}
\frac{R''_\kappa(z(x))}{|R'_\kappa(z(x))|} > \frac{z''(x)}{(z'(x))^2}.
\end{equation*}
 % ------------- %
To estimate the left-hand side from below, we use the explicit form of resolvent kernel. For $\nu = 3$, where $R\kappa(z) = \frac{\e^{-\kappa z}}{4\pi z}$, this is elementary, the expression being
 % ------------- %
\begin{equation*}
\frac{R''_\kappa(z)}{|R'_\kappa(z)|} = \frac{1}{z}\, \frac{1+(1+\kappa z)^2}{1+\kappa z} > \frac{2}{z}.
\end{equation*}
 % ------------- %
For a general $\nu\ge 2$ we have $R_\kappa(z) = (2\pi)^{-\nu/2} \big(\frac{\kappa}{z}\big)^\eta K_\eta(\kappa z)$, where $\eta= \frac{\nu}{2}-1$, and the expression in question can be written as $-\frac{\D}{\D z}\ln|R'_\kappa(z)|$. The derivative can be evaluated explicitly being negative as expected,
 % ------------- %
\begin{align*}
(2\pi)^{\nu/2} & \frac{\D}{\D z} R_\kappa(z) = -\frac{\eta\kappa^\eta}{z^{\eta+1}} K_\eta(\kappa z) + \big(\textstyle{\frac{\kappa}{z}}\big)^\eta K'_\eta(\kappa z) \\
& = -\big(\textstyle{\frac{\kappa}{z}}\big)^\eta \Big[ \frac{\eta}{z} K_\eta(\kappa z) +\kappa K_{\eta-1}(\kappa z) +\frac{\eta}{z} K_\eta(\kappa z)\Big] \\
& = -\kappa\big(\textstyle{\frac{\kappa}{z}}\big)^\eta \Big[ \frac{2\eta}{\kappa z} K_\eta(\kappa z) +K_{\eta-1}(\kappa z) \Big] = -\kappa^{\eta+1} z^{-\eta} K_{\eta+1}(\kappa z),
\end{align*}
 % ------------- %
where we have used successively the relations $K'_\eta(u) = -K_{\eta-1}(u)-\frac{\eta}{u}K_\eta(u)$ and $K_{\eta+1}(u) = K_{\eta-1}(u) +\frac{2\eta}{u}K_\eta(u)$ \cite[9.6.26]{AS}. Differentiating the loga\-rithm of the right-hand side and using the mentioned expression of the Macdonald function derivative again, we get
 % ------------- %
\begin{equation*}
\frac{\D}{\D z}\ln|R'_\kappa(z)| = -\frac{\eta}{z} +\kappa \frac{K'_{\eta+1}(\kappa z)}{K_{\eta+1}(\kappa z)} = -\frac{\nu-1}{z} -\kappa \frac{K_\eta(\kappa z)}{K_{\eta+1}(\kappa z)},
\end{equation*}
 % ------------- %
because $\nu-1=2\eta-1$, and since $\kappa(K_\eta/K_{\eta+1})(\kappa z) > 0$, we arrive at the bound
 % ------------- %
 \begin{equation} \label{conv_bound}
\frac{R''_\kappa(z)}{|R'_\kappa(z)|} = -\frac{\D}{\D z}\ln|R'_\kappa(z)| > \frac{\nu-1}{z}.
\end{equation}
 % ------------- %
On the other hand,
 % ------------- %
\begin{equation*}
\frac{z''(x)}{(z'(x))^2} = \frac{c^2}{z(x)(b+x)^2}\,;
\end{equation*}
 % ------------- %
comparing the two expressions we get the sought bound.
\end{proof}

\noindent \emph{Proof of Theorem~\ref{thm:local_pert_bs}, concluded:} Using the dominated convergence argument again we can interchange the double integral with the sum. Since $\phi_0$ is positive and $V$ is nonnegative and nonzero, it is sufficient to check that the sum of the square brackets in the above explicit expression for the terms of the series \eqref{trial_ineq2} over all pair of mutually different indices $i,j$ is nonnegative any fixed $\xi,\xi'$, and it is positive unless $\eta_{ij}=0$. We decompose the integration variables into the longitudinal and transversal part, $\xi=\xi_1+\xi_\perp$. By assumption we have $|\xi_\perp- \xi'_\perp| \le 2R$, the longitudinal parts are all positive and larger than the smallest distance of the neighboring sets $\Sigma_{\rho,R}(y_i)$ and $\Sigma_{\rho,R}(y_{i+1})$, that is, larger than $2\rho$. We can thus apply Lemma~\ref{l:comp_convex} with $b=2\rho$ and $c=2R$; this concludes the proof.
\end{proof}

 % ------------- %
\begin{remark}
Note that that theorem assumption cover arrays of spherical potential wells and that the restriction becomes weaker for higher dimensions. The condition we obtained for the existence of the discrete spectrum below the threshold of the unperturbed array is, however, sufficient but no means necessary, for several reasons. First of all, the estimate made in the proof of Lemma~\ref{l:comp_convex} is rather rough: we neglected the term $\kappa(K_\eta/K_{\eta+1})(\kappa z)$ which is not only positive but larger than $\kappa$ which means that the bound in the theorem can be replaced, e.g., by $R\le \rho\sqrt{\nu-1+\kappa\rho}$, and since we consider here only $\kappa$ such that $-\kappa^2< \inf\sigma_\mathrm{ess}(H_{V,Y_0})$, the stronger the potential $V$ is, the weaker is the restriction on the aspect ratio of the potential support. Secondly, if $R_1$ satisfies the assumptions of the theorem so that the integral over transversal variables with $|\xi_\perp- \xi'_\perp| \le 2R_1$ produces a positive contribution to the expression \eqref{trial_ineq2} and the actual $R$ is larger than $R_1$, the said positive quantity may still outweigh the remaining part, especially if $R-R_1$ is small. We believe that the restriction we have obtained comes from our proof method and could be disposed of.
\end{remark}
 % ------------- %

%%%%%%%%%%%%%%%%%%%%%%%%%%%%%%%%%%%
\section{Other local perturbations}
\label{s:other}

While some local perturbations moving the points of $Y_0$ away from the array axis may preserve the spectral threshold of $H_{V,Y}$ as shown in Example~\ref{ex:puretrans}, in general they may give rise to bound states as long as their transverse component is not too large.
 % ------------- %
 \begin{proposition} \label{prop:skew}
To a given local longitudinal perturbation $Y\ne Y_0$ there is an $\varepsilon=\varepsilon(Y)$ such that $\sigma_\mathrm{disc}(H_{V,Y'})\ne\emptyset$ provided that $|y_{j,\perp} -y'_{j,\perp}|<\varepsilon$ holds for all $y_j\in P$.
 \end{proposition}
 % ------------- %
\begin{proof}
The expression \eqref{trial_ineq2} is a continuous function of the positions of the points of $Y$, hence being positive for $Y$, it remains to be such for small enough transversal shifts of the point of $Y\setminus Y_0$.
\end{proof}

In reality, however, the bound state existence extends beyond this perturbative regime. To illustrate this, we present the result of a numerical computation in which an equidistant array of two-dimensional circular wells of the shape given by
 % ------------- %
 \begin{equation} \label{well_example}
V(x) = -V_0\, \exp\!\Big\{-\frac{1}{2\sigma^2}|x|^2\Big\}\, \chi_{|x|\le R}(x)
\end{equation}
 % ------------- %
with the parameters $\sigma=0.5$, $V_0 = 5$, $R=1$, and the well center spacing $a=5$, is perturbed by a displacement of one of them. Figure~\ref{fig:moving_well}
 % ------------- %
 \begin{figure}[t]
 \centering
 \includegraphics[scale=0.50]{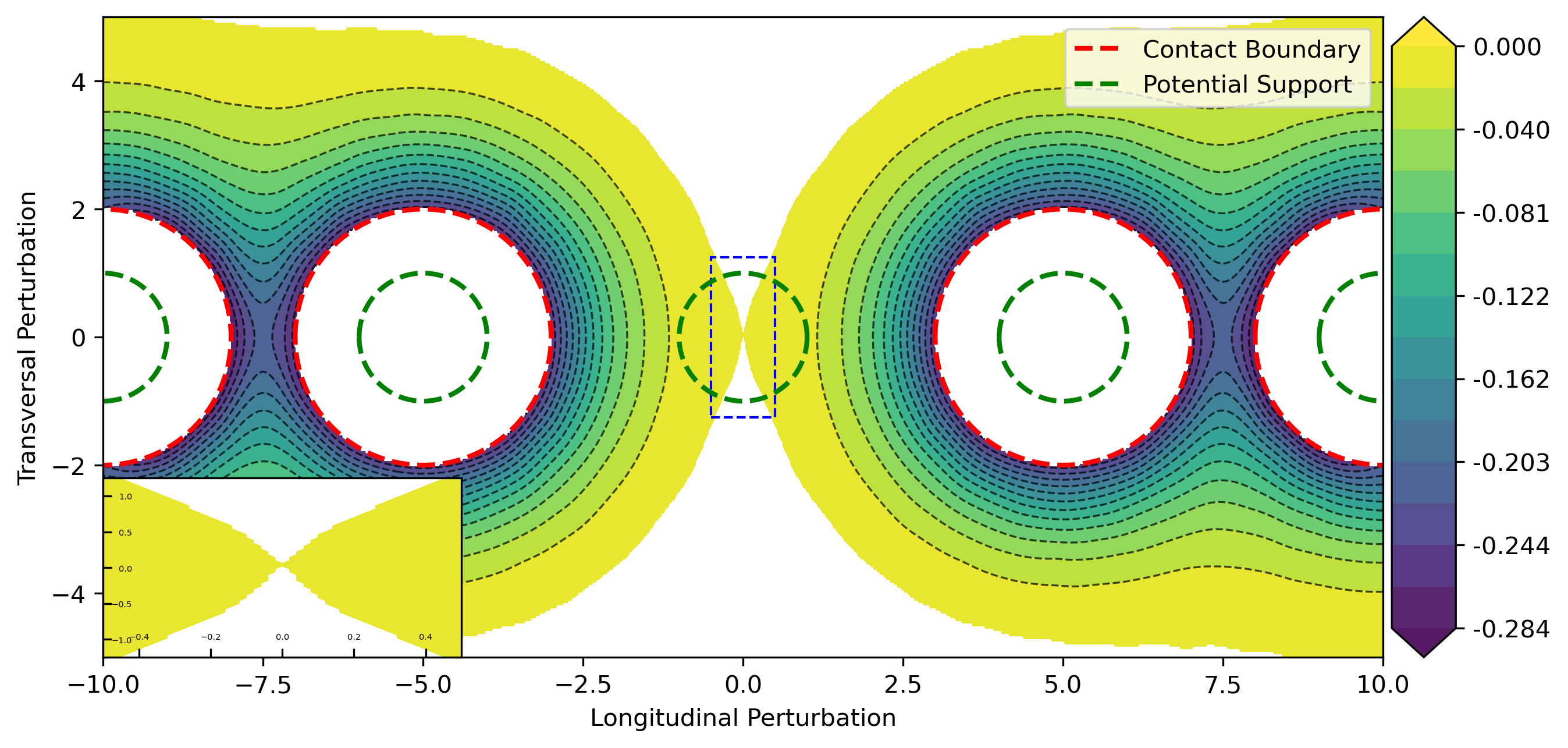}
 \caption{Bound state resulting from a displacement of one potential well with the indication of the binding energy; in the inset the longitudinally expanded neighborhood of the coordinate origin is shown}
 \label{fig:moving_well}
 \end{figure}
 % ------------- %
shows the positions of the displaced well center which gives rise to a bound state with the indication of the corresponding binding energy. The computation was done on the finite array of eleven wells; imposing Dirichlet and Neumann conditions at the ends of the chain we were able to ensure that the error coming from the cut-off is already negligible.

The boundary of the eigenvalue existence area has a linear asymptotics at the unperturbed position as the inset shows. We note also that, while it goes beyond our assumptions and we do not show it here, the discrete spectrum may exist also if the wells can overlap, especially in the situation when the potential on their intersection is the sum of the constituents.

%%%%%%%%%%%%%%%%%%%%%%%%%%%%
\section{Concluding remarks}
\label{s:concl}

The existence result we have obtained is for sure not optimal. As we have mentioned we expect that the aspect ratio bound is not necessary, and one can also conjecture that the result will carry over to situations where the family of local shift parameters $\{\delta_j\}$ will decay sufficiently fast instead of being of compact support. Another hypothesis which is most probably unnecessary is the vertical slicing of the well array into parallel copies of the slab $S_a$. The periodicity cell need not have such a simple geometry, and in dimensions $\nu\ge 3$ it even need not be simply connected; imagine the potential supports in the form of an interlocked chain of annular tubes.

The existence of eigenvalues at the spectrum bottom is naturally not the only question one can ask. Other points of interest concern eigenvalues in the spectral gaps of $H_{V,Y_0}$ and their properties, in particular, their dependence on the `individual' potential $V$. The same questions can be asked in situations when $Y$ is not an array but a lattice in $\R^\nu$, the case $\nu=2$ being of a particular interest. True, we have the results mentioned in the opening, but they do not apply here: local shifts in the potential well lattice are not suited for solution by separation of variables as in \cite{FK98}, and as potential perturbations in sense of \cite{ADH89} they are sign indefinite.

%%%%%%%%%%%%%%%%%%%%%%%%%%%%%%
\begin{funding}
The work was partially supported by the European Union's Horizon 2020 research and innovation programme under the Marie Sk\l odowska-Curie grant agreement No 873071.
\end{funding}


\begin{thebibliography}{99}

 % -------------- %
\bibitem[AS]{AS}
M.S.~Abramowitz, I.A.~Stegun, eds.: \emph{Handbook of Mathematical Functions}, Dover, New York 1965.
 % -------------- %
\bibitem[AF03]{AF03}
R.A.~Adams, J.J.F.~Fournier: \emph{Sobolev Spaces}, 2nd ed., Academic Press, New York, 2003.
 % -------------- %
\bibitem[AS82]{AS82}
M.~Aizenman, B.~Simon: Brownian motion and Harnack inequality for Schr\"odinger operators, \emph{Commun. Pure Appl. Math.} \textbf{35} (1982), 209--273.
 % -------------- %
\bibitem[ADH89]{ADH89}
S.~Alama, P.A.~Deift, R.~Hempel: Eigenvalue branches of the Schr\"odinger operator $H-\lambda W$ in a gap of $\sigma(H)$, \emph{Commun. Math. Phys.} \textbf{121} (1989), 291--321.
 % -------------- %
\bibitem[Ba52]{Ba52}
V.~Bargmann: On the number of bound states in a central field of force, \emph{Proc. Nat. Acad. Sci. U.S.A.} \textbf{38} (1952)
961--966.
 % -------------- %
\bibitem[BEG22]{BEG22}
J.~Behrndt, A.F.M. ter Elst, F.~Gesztesy: The generalized Birman-Schwinger principle, \emph{Trans. Am. Math. Soc.} \textbf{375} (2022), 799--845.
 % -------------- %
\bibitem[BGRS97]{BGRS97}
W.~Bulla, F.~Gesztesy, W.~Renger, B.~Simon: Weakly coupled bound states in quantum waveguides, \emph{Proc. Amer. Math. Soc.} \textbf{127} (1997), 1487--1495.
 % -------------- %
\bibitem[Co97]{Co97}
J.B.~Convay: \emph{A Course in Functional Analysis}, 2nd. ed., Springer, New York 1997
 % -------------- %
\bibitem[DH86]{DH86}
P.A.~Deift, R.~Hempel: On the existence of eigenvalues of the Schr\"odinger operator $H-\lambda W$ in a gap of $\sigma(H)$, \emph{Commun. Math. Phys.} \textbf{103} (1986), 461--490.
 % -------------- %
\bibitem[Ex24]{Ex24}
P.~Exner: Geometry effects in quantum dot families, \emph{Pure Appl. Funct. Anal.} \textbf{9} (2024), 1065--1080.
 % -------------- %
\bibitem[EV24]{EV24}
P.~Exner, S.~Vugalter: Bound states in bent soft waveguides, \emph{J. Spect. Theory} \textbf{14} (2024), 427--457.
 % -------------- %
\bibitem[Fa89]{Fa89}
S.~Fassari: Spectral properties of the Kronig-Penney Hamiltonian with a localised impurity, \emph{J. Math. Phys.} \textbf{30} (1989), 1385--1392.
 % -------------- %
\bibitem[FK98]{FK98}
S.~Fassari, M.~Klaus: Coupling constant thresholds of perturbed periodic Hamiltonians, \emph{J. Math. Phys.} \textbf{39} (1998), 4369--4416.
 % -------------- %
\bibitem[GS93]{GS93}
F.~Gesztesy, B.~Simon: A short proof of Zheludev's theorem, \emph{Trans. Am. Math. Soc.} \textbf{335} (1991), 329--340.
 % -------------- %
\bibitem[JP51]{JP51}
R.~Jost, A.~Pais: On the scattering of a particle by a static potential, \emph{Phys. Rev.} \textbf{82} (1951), 840--850.
 % -------------- %
\bibitem[KS79]{KS79}
M.~Klaus, B.~Simon: Binding of Schr\"odiner particles through conspiracy of potental wells, \emph{Ann. Inst. H.~poincar\'{e}}, \textbf{A30} {1979}, 83--87.
 % -------------- %
\bibitem[RS]{RS}
M.~Reed, B.~Simon: \emph{Methods of Modern Mathematical Physics, II.~Fourier Analysis. Self-adjointness, IV.~Analysis of Operators}, Academic Press, New York 1975, 1978.
 % -------------- %
\bibitem[Se71]{Se71}
Z.~Semadeni: \emph{Banach Spaces of Continuous Function}, Monografie Matematyczne, vol.~55, PWN, Warsaw 1971.
 % -------------- %
\bibitem[Si76]{Si76}
B.~Simon: The bound state of weakly coupled Schr\"odinger ope\-rators in one and two dimensions, \emph{Ann. Phys.} \textbf{97} (1976), 279--288.
 % -------------- %
\bibitem[Si82]{Si82}
B.~Simon: Schr\"odinger semigroups, \emph{Bull. Am. Math. Soc.} \textbf{7} (1982), 447--526.
 % -------------- %
\bibitem[Zh71]{Zh71}
V.A.~Zheludev: Perturbation of the spectrum of the one-dimensional Schr\"odinger operator with a periodic potential, in
\emph{Topics in Mathematical Physics, vol.~4} (M.Sh.~Birman, ed.), Consultants Bureau, New York 1971; pp.~55--75.
 % -------------- %
\end{thebibliography}
\end{document}